\documentclass[11pt]{article}
\usepackage[a4paper,margin=1in]{geometry}
\usepackage[T1]{fontenc}
\usepackage{amsmath,amsthm,mathtools}
\usepackage[tt=false,type1=true]{libertine}
\usepackage[varqu]{zi4}
\usepackage[libertine]{newtxmath}
\usepackage{microtype}
\usepackage[hidelinks]{hyperref}
\usepackage{enumitem}
\usepackage{booktabs,array}

\newtheorem*{restatedinner}{\restatedname}
\newenvironment{restated}[2]{\def\restatedname{#1~\ref{#2}}\begin{restatedinner}}{\end{restatedinner}}
\newtheorem{theorem}{Theorem}
\newtheorem{corollary}[theorem]{Corollary}
\newtheorem{proposition}[theorem]{Proposition}
\newtheorem{lemma}[theorem]{Lemma}
\theoremstyle{definition}
\newtheorem{definition}[theorem]{Definition}
\newtheorem{example}[theorem]{Example}

\newcommand{\Sym}{\mathrm{Sym}}
\newcommand{\ZZ}{\mathbb{Z}}
\newcommand{\cB}{\mathcal{B}}
\newcommand{\cA}{\mathcal{A}}
\newcommand{\cC}{\mathcal{C}}
\newcommand{\cH}{\mathcal{H}}
\newcommand{\cF}{\mathcal{F}}
\newcommand{\cM}{\mathcal{M}}
\newcommand{\ord}{\operatorname{ord}}
\newcommand{\rB}{\mathbf{B}}
\newcommand{\EE}{\mathbb{E}}

\title{\bfseries A $G$-Invariant Erd\H{o}s--Hanani Theorem\\ with Applications to Optical Orthogonal and Constant-Weight Codes}
\author{Yeow Meng Chee\\[2pt]
\small Singapore University of Technology and Design}
\date{}

\begin{document}
\maketitle

\begin{abstract}
Let $k>t\ge 2$ be fixed integers, and let $G$ be a permutation group on a set
of $v$ points. A $G$-invariant packing is free if every block has $|G|$
distinct images under $G$. We give two conditions on $G$ under which, as
$v\to\infty$, there is a free $G$-invariant $t$-$(v,k,1)$ packing with
$(1-o(1))\binom{v}{t}/\binom{k}{t}$ blocks. This is a $G$-invariant form of
the Erd\H{o}s--Hanani theorem on asymptotically optimal packings. For semiregular
groups $G$, packings with these properties exist if and only if all but
$o(v^t)$ of the $t$-subsets of points have trivial setwise stabiliser. This
property is satisfied by every semiregular group for $t\ge 3$, and by every
cyclic semiregular group for $t=2$. As a consequence, the Johnson bound is
asymptotically attained at all lengths by optical orthogonal codes of
constant weight $w$ and constant correlation $\lambda$ with $w\ge\lambda+2$,
by their multidimensional and signature pattern versions, and by cyclic and
quasi-cyclic constant-weight codes. The same method gives a form of the
theorem for group divisible packings that are invariant under a cyclic
group. This determines the asymptotic maximum size of two-dimensional optical
orthogonal codes with at most one pulse per wavelength.
\end{abstract}


\section{Introduction}\label{sec:introduction}

\emph{Optical orthogonal codes} (OOCs) were introduced by Chung, Salehi and Wei
\cite{ChungSalehiWei} as signature sequences for optical code-division
multiple access. Let $v>w>\lambda\ge 1$ be integers. A \emph{$(v,w,\lambda)$ OOC}
is a set $\cC\subseteq\{0,1\}^{v}$ of sequences, called \emph{codewords},
each of which has exactly $w$ entries equal to $1$. We write a codeword as
$x=(x_0,\dots,x_{v-1})$, and we take indices modulo $v$. The set $\cC$ is
required to have the following two properties. The \emph{autocorrelation
property} is
\begin{equation}\label{eq:auto}
 \sum_{i=0}^{v-1}x_ix_{i+a}\le\lambda
 ,\qquad\text{for all } x\in\cC \text{ and all } a\in\ZZ_v\setminus\{0\}.
\end{equation}
The \emph{cross-correlation property} is
\begin{equation}\label{eq:cross}
 \sum_{i=0}^{v-1}x_iy_{i+a}\le\lambda
 ,\qquad\text{for all distinct } x,y\in\cC \text{ and all } a\in\ZZ_v.
\end{equation}
The autocorrelation property allows a receiver to synchronise, and the
cross-correlation property allows it to separate users. Let $\Phi(v,w,\lambda)$ be the maximum number of codewords
in a $(v,w,\lambda)$ OOC.
The \emph{Johnson bound} \cite{ChungSalehiWei} states that
\begin{equation}\label{eq:johnson}
 \Phi(v,w,\lambda)\le\left\lfloor\frac{1}{w}\left\lfloor\frac{v-1}{w-1}
 \left\lfloor\frac{v-2}{w-2}\cdots\left\lfloor\frac{v-\lambda}{w-\lambda}
 \right\rfloor\cdots\right\rfloor\right\rfloor\right\rfloor .
\end{equation}

Omrani, Moreno and Kumar \cite{Omrani} improved \eqref{eq:johnson} for
certain ranges of the parameters in which $\lambda$ is large.
For $\lambda=1$, optimal codes are known for small weights
and many lengths, and for larger weights when the length lies in particular
congruence classes or has a particular algebraic form
\cite{FujiHaraMiao,HuczynskaNg,Yin}. For $\lambda=2$ and $w=4$, optimal codes
are known for infinitely many lengths \cite{FengChangJi}. For general
$\lambda\ge 2$, the known asymptotically optimal families come from finite geometries and finite fields. The lengths of these families have the form $v=q^{a}-1$ or
$v=(q^{a}-1)/(q-1)$, for a prime power $q$, or a similar algebraic form
\cite{AldersonMellinger,Moreno,Omrani}. Lower bounds of Gilbert--Varshamov
type hold for all lengths \cite{ChungSalehiWei}. Two-dimensional and multidimensional
OOCs \cite{Alderson,Omrani2D} and optical orthogonal
signature pattern codes \cite{JiDingWangGe,Kitayama} have analogous Johnson
bounds, and optimal or asymptotically optimal codes are known for particular
parameters. Whether the Johnson bound \eqref{eq:johnson} is asymptotically
attained at every length, for constant $w$ and $\lambda$, has remained an
open problem.

Let $V$ be a set of $v$ points and let $k\ge t\ge 1$ be
integers. A \emph{$t$-$(v,k,1)$ packing} is a family $\cB$ of $k$-subsets of
$V$, called \emph{blocks}, such that every $t$-subset of $V$ lies in at most
one block. Counting $t$-subsets gives
\begin{equation}\label{eq:trivial}
 |\cB|\le \binom{v}{t}\Big/\binom{k}{t}.
\end{equation}
We identify a codeword $x$ with its \emph{support} $C=\{i\in\ZZ_v : x_i=1\}$, which
is a $w$-subset of $\ZZ_v$. For $a\in\ZZ_v$, put $C+a=\{i+a : i\in C\}$.
Then \eqref{eq:auto} and \eqref{eq:cross} hold if and only if the sets $C+a$,
where $C$ is the support of a codeword and $a\in\ZZ_v$, are pairwise distinct
and form a $(\lambda+1)$-$(v,w,1)$ packing
\cite{FujiHaraMiao,Yin}. Let $G$ be a permutation group on $V$. We call a $G$-invariant packing
\emph{free} if every block has $|G|$ distinct images under $G$. The packing
formed by the sets $C+a$ is free and $\ZZ_v$-invariant, where $\ZZ_v$ acts
on itself by translation, and it has $v|\cC|$ blocks. With $t=\lambda+1$ and
$k=w$, the bound \eqref{eq:trivial} therefore gives
\begin{equation}\label{eq:ooc-upper}
 \Phi(v,w,\lambda)\ \le\ \frac{1}{v}\binom{v}{\lambda+1}\Big/\binom{w}{\lambda+1}
 =\frac{(v-1)(v-2)\cdots(v-\lambda)}{w(w-1)\cdots(w-\lambda)}.
\end{equation}
The bound \eqref{eq:johnson} is obtained from \eqref{eq:ooc-upper} by taking
integer parts at each step, and the two bounds agree asymptotically. We refer
to both \eqref{eq:johnson} and \eqref{eq:ooc-upper} as the Johnson bound.
Hence the Johnson bound is asymptotically attained if and only if there are
free $\ZZ_v$-invariant $(\lambda+1)$-$(v,w,1)$ packings with
$(1-o(1))\binom{v}{\lambda+1}/\binom{w}{\lambda+1}$ blocks. Multidimensional
OOCs correspond to free packings that are invariant under a cyclic group,
which acts by translation of one coordinate of a product. Optical orthogonal
signature pattern codes correspond to free packings that are invariant under
a product of two cyclic groups, which acts on itself by translation. For OOCs,
multidimensional OOCs and optical orthogonal signature pattern codes, the
group acts \emph{semiregularly}, that is, no element of the group other than
the identity maps a point to itself.

For packings with no prescribed group of automorphisms, the asymptotic
attainment of the bound \eqref{eq:trivial} is known. Erd\H{o}s and Hanani
\cite{ErdosHanani} conjectured that \eqref{eq:trivial} is
asymptotically attained when $k$ and $t$ are constant and $v\to\infty$, and
R\"odl \cite{Rodl} proved the conjecture. Frankl and R\"odl \cite{FranklRodl}
and Pippenger and Spencer \cite{PippengerSpencer} placed the result in the
setting of matchings in hypergraphs with small codegrees. Keevash
\cite{Keevash} later proved that equality holds in \eqref{eq:trivial} for all
large $v$ that satisfy the divisibility conditions. These theorems do not provide packings that are invariant under a
prescribed group. The asymptotic attainment of the Johnson bound for OOCs is a
case of the problem to decide whether the bound \eqref{eq:trivial} is
asymptotically attained by free $G$-invariant packings, for a prescribed
permutation group $G$ on $V$.

For semiregular groups $G$, we determine when the bound \eqref{eq:trivial}
is asymptotically attained by free $G$-invariant packings. Let $G$ be a
semiregular permutation group on $V$, and let $k>t\ge 2$ be constant. For $g\in G$ and
$A\subseteq V$, let $gA$ denote the set $\{gx : x\in A\}$. Let $\nu_t(G)$ be the number of
$t$-subsets $T$ of $V$ such that $gT=T$, for some element $g$ of $G$ other
than the identity. We prove that, for every
$\varepsilon>0$ and all large $v$, there is a free $G$-invariant $t$-$(v,k,1)$
packing with at least
\[
 \Bigl((1-\varepsilon)\binom{v}{t}-c\,\nu_t(G)\Bigr)\Big/\binom{k}{t}
\]
blocks (Theorem~\ref{thm:geh}). The number $c$ depends only on $k$ and $t$.
We also prove that a free $G$-invariant
$t$-$(v,k,1)$ packing has at most $\bigl(\binom{v}{t}-\nu_t(G)\bigr)/\binom{k}{t}$
blocks. Hence the bound \eqref{eq:trivial} is asymptotically attained by
free $G$-invariant packings if and only if $\nu_t(G)=o(v^t)$. This
condition holds for every semiregular group if $t\ge 3$. For $t=2$, we have
$\nu_2(G)=\iota(G)v/2$, where $\iota(G)$ is the number of elements of order
two in $G$, and the condition is $\iota(G)=o(v)$. The group $G$ may be
nonabelian, and its order may be any divisor of $v$. The groups that arise
from OOCs, multidimensional OOCs and optical orthogonal signature pattern
codes have at most three elements of order two. Consequently, the Johnson
bound is asymptotically attained at all lengths by OOCs, by multidimensional
OOCs and by optical orthogonal signature pattern codes, for constant $w$ and
$\lambda$ with $w\ge\lambda+2$. Moreover, cyclic and quasi-cyclic binary
codes of constant weight $w$ and minimum distance $2\delta$ attain
asymptotically the maximum size of a binary code with these parameters, for
constant $w$ and $\delta$ with $2\le\delta\le w-1$.

A $G$-invariant packing is a union of orbits of $G$ on the $k$-subsets of
$V$. Our proof selects these orbits
by means of a matching in a hypergraph, as in the proofs of Frankl and
R\"odl and of Pippenger and Spencer. The vertices of the hypergraph are the
orbits of the $t$-subsets of $V$ that have $|G|$ distinct images. There is
one edge for each orbit of $k$-subsets $B$ such that $B$ and $gB$ have fewer
than $t$ common points, for all $g\in G$ other than the identity. The edge
consists of the orbits of the $t$-subsets of $B$. A matching with $n$ edges
yields a free $G$-invariant packing with $n|G|$ blocks. The theorem of
Kahn \cite{Kahn} on the chromatic index of hypergraphs with small codegrees
gives a large matching relative to the number of edges and the maximum
degree. Our
application has two tasks, and a nontrivial group is an obstacle to each of
them. The first task is to show that almost all $k$-subsets give edges. A $k$-subset $B$ gives no edge if $B$ and $gB$ have
at least $t$ common points, for some $g\in G$ other than the identity. The second task is to bound the codegrees. For $t$-subsets $T$ and
$S$ of $V$ in distinct orbits, only a small proportion of the $k$-subsets
that include $T$ may include an image of $S$, and $S$ has up to $|G|$
images. We prove a $G$-invariant Erd\H{o}s--Hanani
theorem for an arbitrary permutation group $G$. The lower bound and the hypothesis of this theorem are expressed by two parameters $\alpha_{k,t}(G)$ and
$\beta_{k,t}(G)$, which measure the first and the second obstacle,
respectively. For semiregular groups, we prove that
$\beta_{k,t}(G)=O(1/v)$ and $\alpha_{k,t}(G)=O(1/v+\nu_t(G)/v^t)$.

The Erd\H{o}s--Hanani theorem concerns packings of the $t$-subsets of a set
by $k$-subsets. OOCs of dimension two with at most one pulse per wavelength
correspond to packings of a different kind. The points are the elements of
$\ZZ_\Lambda\times\ZZ_T$. A subset of $\ZZ_\Lambda\times\ZZ_T$ is a
\emph{transversal} if it has at most one point in each of the sets
$\{i\}\times\ZZ_T$. Each block is a transversal $k$-subset, and every
transversal $t$-subset lies in at most one block. A packing with these two
properties is a \emph{group divisible packing}. A group divisible packing has at most
$\binom{\Lambda}{t}T^{t}/\binom{k}{t}$ blocks. The group $\ZZ_T$ acts on
$\ZZ_\Lambda\times\ZZ_T$ by translation of the second coordinate.
We show that there are $\ZZ_T$-invariant group divisible
packings with $(1-o(1))\binom{\Lambda}{t}T^{t}/\binom{k}{t}$ blocks, uniformly in $\Lambda\ge k$ and $T$ (Theorem~\ref{thm:gd}). The proof uses a matching in the hypergraph whose
edges come from the orbits of transversal $k$-subsets.
This result gives the asymptotic maximum size of OOCs with at
most one pulse per wavelength.

\section{Group Actions and Main Results}\label{sec:main}

Throughout, $V$ is a set of $v$ points, $G\le\Sym(V)$ is a permutation
group, and $m=|G|$. We write $1$ for the identity of $G$. For $g\in G$ and
$A\subseteq V$, put $gA=\{gx : x\in A\}$. Thus $G$ acts on the set of all
subsets of $V$. The sets $gA$, $g\in G$, are the \emph{images} of $A$. The set of all images of $A$ is the \emph{orbit} of $A$, and the orbit of a
point $x\in V$ is the set $\{gx : g\in G\}$. The subgroup $\{g\in G : gA=A\}$ is the \emph{stabiliser} of $A$. A subset
$A\subseteq V$ is \emph{free} if the stabiliser of $A$ is trivial. The orbit
of a free set has size $m$. For an integer $t\ge 1$, the number of
$t$-subsets of $V$ that are not free is $\nu_t(G)$. Let $\cA$ be a family of subsets of $V$ that is a union of orbits. A
\emph{set of orbit representatives} of $\cA$ is a subfamily of $\cA$ that
contains exactly one member of each orbit included in $\cA$.

Let $k\ge t\ge 1$ be integers. A packing $\cB$ is \emph{$G$-invariant} if
$gB\in\cB$, for all $B\in\cB$ and all $g\in G$. The packing $\cB$ is free if it is $G$-invariant and every block is a free set. Hence a free $G$-invariant
packing $\cB$ is a union of $G$-orbits of size $m$. The members of a set of orbit representatives of $\cB$ are called
\emph{base blocks}, and their number is $|\cB|/m$. For $G=\ZZ_v$ acting on itself by translation, the free
$G$-invariant $(\lambda+1)$-$(v,w,1)$ packings are the packings that
correspond to $(v,w,\lambda)$ OOCs. Let $P_G(v,k,t)$ be the maximum number of
blocks in a free $G$-invariant $t$-$(v,k,1)$ packing on $V$. Every lower
bound for $P_G(v,k,t)$ is a lower bound for the maximum number of blocks in a
$G$-invariant $t$-$(v,k,1)$ packing.

Recall that $G$ is semiregular if $gx\ne x$, for all $x\in V$ and all
$g\in G\setminus\{1\}$. If $G$ is semiregular, then every orbit of $G$ on $V$
has size $m$, so $m$ divides $v$. For $x,y\in V$ there is at most one
$g\in G$ with $gx=y$, and every cycle of an element $g\in G\setminus\{1\}$ on
$V$ has length $\ord(g)$. A semiregular group is \emph{regular} if $m=v$. An
\emph{involution} is an element of order two, and $\iota(G)$ is the number of
involutions in $G$.

\begin{lemma}\label{lem:nu}
Let $G$ be semiregular. Then $\nu_2(G)=\iota(G)v/2$, and
$\nu_t(G)\le v^{t-1}$, for $t\ge 3$.
\end{lemma}

\begin{proof}
A $2$-subset $\{x,y\}$ fails to be free if and only if some $g\in G\setminus\{1\}$
interchanges $x$ and $y$. Such an element is an involution. Each involution
$g$ yields the $v/2$ pairs $\{x,gx\}$. Different involutions yield different
pairs, because at most one element of $G$ maps $x$ to $gx$. Hence
$\nu_2(G)=\iota(G)v/2$.

Let $t\ge 3$, and let $T$ be a $t$-subset of $V$ with $gT=T$, for some
$g\in G\setminus\{1\}$. Then $T$ is a union of cycles of $g$, and each cycle has length
$\ord(g)\ge 2$. Hence $T$ is determined by $g$ and by one point from each
of these cycles. There are at most $m\le v$ elements $g$ and at most
$\lfloor t/2\rfloor$ cycles. Therefore
$\nu_t(G)\le v^{1+\lfloor t/2\rfloor}\le v^{t-1}$.
\end{proof}

 No block of a free
$G$-invariant $t$-$(v,k,1)$ packing contains a nonfree $t$-subset. This gives the following upper bound.

\begin{proposition}\label{prop:pairs}
Let $k\ge t\ge 1$ be integers and let $G\le\Sym(V)$. Then
\[
 P_G(v,k,t)\ \le\ \Bigl(\binom{v}{t}-\nu_t(G)\Bigr)\Big/\binom{k}{t}.
\]
\end{proposition}

The bound in Proposition~\ref{prop:pairs} holds for free packings only,
because a block that is not free may include a $t$-subset that is not free.
A $k$-subset $B$ of $V$ is a block of some free $G$-invariant $t$-$(v,k,1)$
packing only if
\begin{equation}\label{eq:sep}
 |B\cap gB|<t,\qquad\text{for all } g\in G\setminus\{1\}.
\end{equation}
Indeed, the images $gB$, $g\in G$, are then pairwise distinct blocks of the
packing. For every
$k$-subset $B$ of $V$, put
\[
 N(B)=\sum_{g\in G\setminus\{1\}}\binom{|B\cap gB|}{t}.
\]
Thus $N(B)$ is the number of pairs $(g,T)$ such that $g\in G\setminus\{1\}$
and $T$ is a $t$-subset of $B\cap gB$, and $N(B)=0$ if and only if $B$
satisfies \eqref{eq:sep}. If $N(B)=0$, then the images $gB$, $g\in G$, are
pairwise distinct and form a free $G$-invariant $t$-$(v,k,1)$ packing.
Let $\rB$ be
a $k$-subset of $V$ chosen uniformly at random. Throughout, the bold letter
$\rB$ denotes this random set, and $B$ denotes a particular $k$-subset of
$V$. Put
\begin{equation}\label{eq:alpha}
 \alpha_{k,t}(G)=\EE\bigl[N(\rB)\bigr].
\end{equation}
Since $N(\rB)$ is a nonnegative integer, Markov's inequality gives
\[
 \Pr\bigl[N(\rB)\ne 0\bigr]=\Pr\bigl[N(\rB)\ge 1\bigr]
 \le\EE\bigl[N(\rB)\bigr]=\alpha_{k,t}(G).
\]
Hence $N(B)=0$, for all but at most $\alpha_{k,t}(G)\binom{v}{k}$ of the
$k$-subsets $B$ of $V$.

Let $T$ be a $t$-subset of $V$ that is not free, and let
$g\in G\setminus\{1\}$ with $gT=T$. If $T\subseteq B$, then
$T\subseteq B\cap gB$. Hence $N(B)$ is at least the number of $t$-subsets of
$B$ that are not free. Since
$\Pr[T\subseteq\rB]=\binom{k}{t}/\binom{v}{t}$, we obtain
\begin{equation}\label{eq:alpha-nu}
 \alpha_{k,t}(G)\ \ge\ \binom{k}{t}\,\nu_t(G)\Big/\binom{v}{t}.
\end{equation}
Consequently, if $\nu_t(G)\ge c\binom{v}{t}$, for a constant $c>0$, then
$P_G(v,k,t)\le(1-c)\binom{v}{t}/\binom{k}{t}$ by
Proposition~\ref{prop:pairs}, and $\alpha_{k,t}(G)\ge c$ by
\eqref{eq:alpha-nu}.

Let $T$ and $S$ be $t$-subsets of $V$ in distinct orbits of $G$. Our proof requires that, among
the $k$-subsets of $V$ that include $T$, only a small proportion also include
an image of $S$. For $t$-subsets $T$ and $S$ of $V$ and a $k$-subset $B$, let
$M_{T,S}(B)$ be the number of elements $g\in G$ such that $gS\subseteq B$ and
$gS\ne T$. Put
\begin{equation}\label{eq:beta}
 \beta_{k,t}(G)=\max_{T,S}\ \EE\bigl[M_{T,S}(\rB)\,\big|\,T\subseteq\rB\bigr],
\end{equation}
where $T$ and $S$ range over the $t$-subsets of $V$. Let $T$ and $S$ lie in
distinct orbits. Then $gS\ne T$, for all $g\in G$, and so $M_{T,S}(B)\ge 1$
for every $k$-subset $B$ that includes an image of $S$. There are
$\binom{v-t}{k-t}$ $k$-subsets of $V$ that include $T$, and $\rB$ is uniformly
distributed over them when $T\subseteq\rB$. Hence the number of $k$-subsets
of $V$ that include $T$ and an image of $S$ is at most
\begin{equation}\label{eq:codeg}
 \sum_{B\supseteq T}M_{T,S}(B)
 =\binom{v-t}{k-t}\,\EE\bigl[M_{T,S}(\rB)\,\big|\,T\subseteq\rB\bigr]
 \le\beta_{k,t}(G)\binom{v-t}{k-t}.
\end{equation}

\begin{theorem}[$G$-invariant Erd\H{o}s--Hanani theorem]\label{thm:general}
Let $k>t\ge 2$ be integers and let $\varepsilon>0$. There exists a number
$\delta=\delta(k,t,\varepsilon)>0$ with the following property. Let $V$ be a
set of $v\ge k$ points, and let $G\le\Sym(V)$ be a permutation group with
$\beta_{k,t}(G)\le\delta$. Then
\[
 P_G(v,k,t)\ \ge\ \bigl(1-\varepsilon-\alpha_{k,t}(G)\bigr)
 \binom{v}{t}\Big/\binom{k}{t}.
\]
\end{theorem}

For the trivial group, $\alpha_{k,t}(G)=0$ and $\beta_{k,t}(G)\le k/v$, and
Theorem~\ref{thm:general} is R\"odl's theorem. For semiregular groups,
Theorem~\ref{thm:general} gives a lower bound of the same form as the upper
bound in Proposition~\ref{prop:pairs}.

\begin{theorem}[Semiregular groups]\label{thm:geh}
Let $k>t\ge 2$ be integers and let $\varepsilon>0$. There exist an integer
$v_0=v_0(k,t,\varepsilon)$ and a number $c=c(k,t)>0$ with the following
property. Let $G$ be a semiregular permutation group on a set of $v\ge v_0$
points. Then
\[
 P_G(v,k,t)\ \ge\ \Bigl((1-\varepsilon)\binom{v}{t}-c\,\nu_t(G)\Bigr)
 \Big/\binom{k}{t}.
\]
\end{theorem}

By Proposition~\ref{prop:pairs} and Theorem~\ref{thm:geh}, a sequence of
semiregular groups $G$ of degree $v\to\infty$ satisfies
$P_G(v,k,t)=(1-o(1))\binom{v}{t}/\binom{k}{t}$ if and only if
$\nu_t(G)=o(v^t)$. By Lemma~\ref{lem:nu}, this condition holds for every
semiregular group if $t\ge 3$, and for $t=2$ it is the condition
$\iota(G)=o(v)$. The following example shows that the condition
$\nu_2(G)=o(v^2)$ fails for some regular groups.

\begin{example}\label{ex:elementary}
Let $G=\ZZ_2^{a}$ act on itself by translation, so that $v=2^{a}$. Every
$2$-subset $\{x,y\}$ of $G$ is mapped to itself by the translation by $x+y$.
Hence $\nu_2(G)=\binom{v}{2}$, and
Proposition~\ref{prop:pairs} gives $P_G(v,k,2)=0$, for all $k\ge 2$. Let $k\ge 3$ be odd, and
let $B$ be a block of a $G$-invariant $2$-$(v,k,1)$ packing. The stabiliser of $B$ is a $2$-group. The order of the stabiliser divides $k$, because $B$ is a union of orbits of its stabiliser on $V$, and each of these orbits has as many points as the stabiliser has elements. Hence the stabiliser of $B$ is trivial. Thus the packing is free, and it is
empty by Proposition~\ref{prop:pairs}. Therefore every
$G$-invariant $2$-$(v,k,1)$ packing is empty, whereas $2$-$(v,k,1)$ packings
with $(1-o(1))\binom{v}{2}/\binom{k}{2}$ blocks exist.
\end{example}

\section{Orbit Hypergraphs}\label{sec:orbit}

A \emph{hypergraph} $\cH$ consists of a vertex set and a family of edges,
each of which is a nonempty set of vertices. Repeated edges are allowed. The
\emph{degree} of a vertex is the number of edges that contain it. The
\emph{codegree} of two distinct vertices is the number of edges that contain
both. A \emph{matching} is a set of pairwise disjoint edges. We use the
following theorem of Kahn \cite{Kahn}.

\begin{theorem}[Kahn]\label{thm:kahn}
For every integer $s\ge 2$ and every $\varepsilon>0$ there exists $\delta>0$
with the following property. Let $\cH$ be a hypergraph in which every edge
has at most $s$ vertices, every vertex has degree at most $D$, and every two
distinct vertices have codegree at most $\delta D$. Then the edges of $\cH$
can be partitioned into at most $(1+\varepsilon)D$ matchings. In particular,
$\cH$ has a matching with at least $|E(\cH)|/((1+\varepsilon)D)$ edges.
\end{theorem}

Theorem~\ref{thm:kahn} follows from the statement for list colourings in
\cite{Kahn}. See also \cite[Theorem~2.4]{Survey}.

Let $k\ge t\ge 2$ be integers. A $k$-subset $B$ of $V$ is \emph{$t$-separated}
if it satisfies \eqref{eq:sep}, that is, if $|B\cap gB|\le t-1$, for all
$g\in G\setminus\{1\}$.

\begin{lemma}\label{lem:sep-basic}
Let $B$ be a $t$-separated $k$-subset of $V$.
\begin{enumerate}[label=\textup{(\roman*)},leftmargin=*]
\item The set $B$ is free, and every $t$-subset of $B$ is free.
\item Distinct $t$-subsets of $B$ lie in distinct $G$-orbits.
\item If $T$ is a $t$-subset of $B$ and $T\subseteq gB$, for some $g\in G$,
then $g=1$.
\end{enumerate}
\end{lemma}

\begin{proof}
Let $g\in G\setminus\{1\}$, and let $T$ be a $t$-subset of $B$. Since $B$
is $t$-separated, $|B\cap gB|\le t-1$.

(i) If $gB=B$, then $|B\cap gB|=k\ge t$. If $gT=T$, then
$T\subseteq B\cap gB$, and so $|B\cap gB|\ge t$. Both conclusions contradict $|B\cap gB|\le t-1$. Hence
$B$ and $T$ are free.

(ii) If $gT\subseteq B$, then $gT\subseteq B\cap gB$, and so
$|B\cap gB|\ge t$. This is impossible. Hence no $t$-subset of $B$ other than
$T$ lies in the orbit of $T$.

(iii) If $T\subseteq gB$, then $T\subseteq B\cap gB$, and so
$|B\cap gB|\ge t$. This is impossible. Hence $T\subseteq hB$ with $h\in G$
only if $h=1$.
\end{proof}

\begin{lemma}\label{lem:sep-count}
Let $k\ge t\ge 2$ be integers. A $k$-subset $B$ of $V$ is $t$-separated if
and only if $N(B)=0$. The number of $k$-subsets of $V$ that fail to be
$t$-separated is at most $\alpha_{k,t}(G)\binom{v}{k}$.
\end{lemma}

\begin{proof}
Both statements were shown in Section~\ref{sec:main}. We have $N(B)=0$ if
and only if $B$ satisfies \eqref{eq:sep}, and Markov's inequality gives
$N(B)=0$, for all but at most $\alpha_{k,t}(G)\binom{v}{k}$ of the
$k$-subsets $B$ of $V$.
\end{proof}

Let $X$ be the set of $G$-orbits of free $t$-subsets of $V$. For a $t$-subset
or a $k$-subset $A$ of $V$, write $[A]$ for its orbit under $G$. If $B$ is a
$t$-separated $k$-subset and $g\in G$, then $gB$ is $t$-separated, because
$|gB\cap hgB|=|B\cap g^{-1}hgB|$, for all $h\in G$.

\begin{definition}
Let $k>t\ge 2$. The \emph{orbit hypergraph} $\cH_k$ has vertex set $X$ and one edge for each orbit $[B]$ of $t$-separated $k$-subsets of $V$, and
this edge is the set
\[
 \partial[B]=\{[T] : T\subseteq B,\ |T|=t\}.
\]
\end{definition}

The set $\partial[B]$ is independent of the representative $B$. By
Lemma~\ref{lem:sep-basic}, it is a subset of $X$ of size $\binom{k}{t}$.
Different orbits may give the same set of vertices, and such edges are
counted separately.

\begin{lemma}\label{lem:matching-packing}
Let $k>t\ge 2$, and let $\cM$ be a matching in $\cH_k$. Then the union $\cB$
of the orbits that correspond to the edges of $\cM$ is a free $G$-invariant
$t$-$(v,k,1)$ packing with $m|\cM|$ blocks.
\end{lemma}

\begin{proof}
The family $\cB$ is $G$-invariant, and its blocks are free by
Lemma~\ref{lem:sep-basic}. Let $B_1$ and $B_2$ be distinct blocks that both
include a $t$-subset $T$. If $B_2=gB_1$, for some $g\in G$, then $g\ne 1$ and
$T\subseteq B_1\cap gB_1$, which contradicts the separation of $B_1$.
Otherwise $\partial[B_1]$ and $\partial[B_2]$ are distinct edges of $\cM$
and both contain $[T]$, which contradicts disjointness.
\end{proof}

For a set $\cF$ of $k$-subsets of $V$, let $P_G(\cF,t)$ be the maximum
number of blocks in a free $G$-invariant $t$-$(v,k,1)$ packing, all of whose
blocks belong to $\cF$. Thus $P_G(\cF,t)=P_G(v,k,t)$ when $\cF$ is the set
of all $k$-subsets of $V$.

\begin{proposition}[Packings with prescribed blocks]\label{prop:restricted}
Let $k>t\ge 2$ be integers, let $\varepsilon>0$, and let $\delta$ be the
number given by Theorem~\ref{thm:kahn} for $s=\binom{k}{t}$ and
$\varepsilon$. Let $\cF$ be a set of $t$-separated $k$-subsets of $V$ such
that $gB\in\cF$, for all $B\in\cF$ and all $g\in G$. Let $D>0$ be a number
with the following two properties. Every $t$-subset of $V$ is included in at
most $D$ members of $\cF$. For all $t$-subsets $T$ and $S$ of $V$ in distinct
orbits, at most $\delta D$ members of $\cF$ include $T$ and an image of $S$.
Then
\[
 P_G(\cF,t)\ \ge\ \frac{|\cF|}{(1+\varepsilon)D}.
\]
\end{proposition}

\begin{proof}
Let $\cH$ be the hypergraph that consists of the vertices of $\cH_k$ and of
the edges $\partial[B]$ with $B\in\cF$. The members of $\cF$ are free by
Lemma~\ref{lem:sep-basic}, so their orbits have size $m$, and $\cH$ has
$|\cF|/m$ edges.

Let $T$ be a free $t$-subset of $V$. Every edge of $\cH$ that contains $[T]$
has the form $\partial[B]$ with $B\in\cF$ and $T\subseteq B$. By
Lemma~\ref{lem:sep-basic}(iii), the orbit $[B]$ contains exactly one set that
includes $T$. Hence the degree of $[T]$ in $\cH$ is the number of members of
$\cF$ that include $T$, which is at most $D$.

Let $[T]$ and $[S]$ be distinct vertices. Every edge that contains both has
the form $\partial[B]$ with $B\in\cF$ and $T\subseteq B$, and $B$ includes an
image of $S$. Hence the codegree of $[T]$ and $[S]$ is at most $\delta D$.

By Theorem~\ref{thm:kahn}, the hypergraph $\cH$ has a matching $\cM$ with at
least $|\cF|/((1+\varepsilon)mD)$ edges. By
Lemma~\ref{lem:matching-packing}, the union of the orbits that correspond to
the edges of $\cM$ is a free $G$-invariant $t$-$(v,k,1)$ packing with $m|\cM|$ blocks. The blocks of this packing belong to $\cF$.
\end{proof}

Semiregularity controls both $\alpha_{k,t}(G)$ and $\beta_{k,t}(G)$. In a
semiregular group, at most one element maps a given point $x$ to a given
point $y$. Hence, for $t$-subsets $T$ and $S$ of $V$, at most $t^2$ elements
$g\in G$ satisfy $T\cap gS\ne\emptyset$, and for every other element $g$ the
image $gS$ has $t$ points outside $T$. This bounds $\beta_{k,t}(G)$. For
$\alpha_{k,t}(G)$, the pairs $(g,T)$ with $g\ne 1$ and
$T\cap gT\ne\emptyset$ are counted according to the paths and cycles that
$g$ induces on $T$. For $t=2$, an involution $g$ and a $2$-subset $T$ whose
two points are exchanged by $g$ satisfy $gT=T$, and these pairs give the term
$\nu_2(G)/\binom{v}{2}$ in the bound.

\begin{lemma}\label{lem:semireg}
Let $k>t\ge 2$ be integers and let $G$ be semiregular. There is a constant
$c_1=c_1(k,t)\ge 1$ such that
\[
 \alpha_{k,t}(G)\le c_1\Bigl(\frac{1}{v}+\nu_t(G)\Big/\binom{v}{t}\Bigr)
 \qquad\text{and}\qquad
 \beta_{k,t}(G)\le\frac{c_1}{v}.
\]
\end{lemma}

\begin{proof}
Let $J\subseteq V$ with $|J|=j$. If $j\le k$, then
$\Pr[J\subseteq\rB]=\binom{v-j}{k-j}/\binom{v}{k}$, and otherwise this
probability is zero. In both cases
\begin{equation}\label{eq:prob}
 \Pr[J\subseteq\rB]\ \le\ \Bigl(\frac{k}{v}\Bigr)^{j}.
\end{equation}
If $J$ includes a $t$-subset $T$, then likewise
\begin{equation}\label{eq:condprob}
 \Pr[J\subseteq\rB\mid T\subseteq\rB]\ \le\ \Bigl(\frac{k}{v}\Bigr)^{j-t},
\end{equation}
because this probability is $\binom{v-j}{k-j}/\binom{v-t}{k-t}$ if $j\le k$,
and zero otherwise.

A $t$-subset $T'$ lies in $B\cap gB$ if and only if $T=g^{-1}T'$ satisfies
$T\subseteq B$ and $gT\subseteq B$. Hence
\begin{equation}\label{eq:alpha-sum}
 \alpha_{k,t}(G)=\sum_{g\in G\setminus\{1\}}\ \sum_{T}
 \Pr[T\cup gT\subseteq\rB],
\end{equation}
where $T$ ranges over the $t$-subsets of $V$. Let $T$ be a $t$-subset of $V$
and let $g\in G\setminus\{1\}$. Let $\Gamma_g(T)$ be the directed graph on
$T$ with an arc from $x$ to $gx$ whenever $x\in T$ and $gx\in T$. The number
of arcs is $i=|T\cap gT|$, and $|T\cup gT|=2t-i$. Every vertex has at most
one outgoing arc and at most one incoming arc, and there are no loops. Hence
every component of $\Gamma_g(T)$ with at least one arc is a directed path or
a directed cycle. A directed cycle is a cycle of the permutation $g$, so it
has exactly $\ord(g)$ vertices and $\ord(g)$ arcs. Let $p$ be the number of
components that are paths with at least one arc, and let $q$ be the number of
components that are cycles. A path with $j$ arcs has $j+1$ vertices, and a
cycle with $j$ arcs has $j$ vertices. Therefore the components with at least
one arc have $i+p$ vertices in total.

We bound the contribution to \eqref{eq:alpha-sum} of the pairs $(T,g)$ for
which $\Gamma_g(T)$ has a given isomorphism type. The number of types depends
only on $t$. Consider one type and one element $g$. The type determines the number of vertices of each component. A component
that is a directed path is determined by $g$ and by its initial vertex. A
component that is a directed cycle is determined by $g$ and by any one of
its vertices. Hence there are at most $v^{p+q}$ choices for the vertices of
the components with at least one arc, and
at most $v^{t-i-p}$ choices for the remaining points of $T$. Thus at most
$v^{t-i+q}$ sets $T$ arise. By \eqref{eq:prob}, each contributes at most
$(k/v)^{2t-i}$. The contribution of this type and this $g$ is therefore at
most $k^{2t}v^{q-t}$.

If $q=0$, then summation over the at most $v$ elements $g$ gives at most
$k^{2t}v^{1-t}\le k^{2t}v^{-1}$. Suppose that $q\ge 1$. Each cycle has at
least two arcs, so $q\le\lfloor t/2\rfloor$. If $t\ge 3$, then
$\lfloor t/2\rfloor\le t-2$, and summation over $g$ gives at most
$k^{2t}v^{1+q-t}\le k^{2t}v^{-1}$. If $t=2$, then $q=1$, the two arcs form a
cycle of length two, and $g$ is an involution. Summation over the involutions
gives at most $k^{2t}\iota(G)\,v^{-1}$. Since
$\nu_2(G)=\iota(G)v/2$ by Lemma~\ref{lem:nu}, we have $\iota(G)/v\le\nu_2(G)/\binom{v}{2}$. The bound on $\alpha_{k,t}(G)$
follows, with $c_1$ at least $k^{2t}$ times the number of types.

Let $T$ and $S$ be $t$-subsets of $V$. Then
\[
 \EE\bigl[M_{T,S}(\rB)\,\big|\,T\subseteq\rB\bigr]
 =\sum_{\substack{g\in G\\ gS\ne T}}\Pr[T\cup gS\subseteq\rB\mid T\subseteq\rB].
\]
Let $g\in G$ with $gS\ne T$. Then $|T\cup gS|\ge t+1$. If
$gS\cap T\ne\emptyset$, then $gy=x$, for some $y\in S$ and $x\in T$.
Semiregularity gives at most $t^2$ such elements $g$, and each contributes at
most $k/v$ by \eqref{eq:condprob}. If $gS\cap T=\emptyset$, then
$|T\cup gS|=2t$, and $g$ contributes at most $(k/v)^{t}$. There are at most
$m\le v$ such elements, and $v(k/v)^{t}\le k^{t}/v$. Hence
$\beta_{k,t}(G)\le(t^2k+k^{t})/v$.
\end{proof}

\section{Proofs of the Main Results}\label{sec:proofs}

\begin{restated}{Proposition}{prop:pairs}
Let $k\ge t\ge 1$ be integers and let $G\le\Sym(V)$. Then
\[
 P_G(v,k,t)\ \le\ \Bigl(\binom{v}{t}-\nu_t(G)\Bigr)\Big/\binom{k}{t}.
\]
\end{restated}

\begin{proof}
Let $\cB$ be a free $G$-invariant $t$-$(v,k,1)$ packing. Let $T$ be a
$t$-subset of $V$ that is not free, and let $g\in G\setminus\{1\}$ with
$gT=T$. Suppose that a block $B$ includes $T$. Then $gB$ includes $gT=T$.
Since $B$ is free, $gB\ne B$, and so two distinct blocks include $T$. This is
impossible. Hence no block includes a $t$-subset that is not free. The blocks
of $\cB$ therefore include at most $\binom{v}{t}-\nu_t(G)$ of the
$t$-subsets of $V$, and each $t$-subset lies in at most one block.
\end{proof}

\begin{restated}{Theorem}{thm:general}[$G$-invariant Erd\H{o}s--Hanani theorem]
Let $k>t\ge 2$ be integers and let $\varepsilon>0$. There exists a number
$\delta=\delta(k,t,\varepsilon)>0$ with the following property. Let $V$ be a
set of $v\ge k$ points, and let $G\le\Sym(V)$ be a permutation group with
$\beta_{k,t}(G)\le\delta$. Then
\[
 P_G(v,k,t)\ \ge\ \bigl(1-\varepsilon-\alpha_{k,t}(G)\bigr)
 \binom{v}{t}\Big/\binom{k}{t}.
\]
\end{restated}

\begin{proof}
Put $D=\binom{v-t}{k-t}$. Let $\delta$ be the number given by
Theorem~\ref{thm:kahn} for $s=\binom{k}{t}$ and $\varepsilon$, and suppose
that $\beta_{k,t}(G)\le\delta$. Let $\cF$ be the set of all $t$-separated
$k$-subsets of $V$. As shown before the definition of $\cH_k$, we have
$gB\in\cF$, for all $B\in\cF$ and all $g\in G$. By
Lemma~\ref{lem:sep-count}, $|\cF|\ge(1-\alpha_{k,t}(G))\binom{v}{k}$. Every
$t$-subset of $V$ is included in at most $D$ members of $\cF$. By
\eqref{eq:codeg}, for all $t$-subsets $T$ and $S$ in distinct orbits, at most
$\beta_{k,t}(G)D\le\delta D$ members of $\cF$ include $T$ and an image of
$S$. Since $\binom{v}{k}/D=\binom{v}{t}/\binom{k}{t}$,
Proposition~\ref{prop:restricted} gives
\[
 P_G(v,k,t)\ \ge\ P_G(\cF,t)\ \ge\ \frac{1-\alpha_{k,t}(G)}{1+\varepsilon}
 \binom{v}{t}\Big/\binom{k}{t}.
\]
The theorem follows, because $(1-x)/(1+\varepsilon)\ge 1-\varepsilon-x$ for
all $x\ge 0$.
\end{proof}

\begin{restated}{Theorem}{thm:geh}[Semiregular groups]
Let $k>t\ge 2$ be integers and let $\varepsilon>0$. There exist an integer
$v_0=v_0(k,t,\varepsilon)$ and a number $c=c(k,t)>0$ with the following
property. Let $G$ be a semiregular permutation group on a set of $v\ge v_0$
points. Then
\[
 P_G(v,k,t)\ \ge\ \Bigl((1-\varepsilon)\binom{v}{t}-c\,\nu_t(G)\Bigr)
 \Big/\binom{k}{t}.
\]
\end{restated}

\begin{proof}
Let $\delta=\delta(k,t,\varepsilon/2)$ be the number given by
Theorem~\ref{thm:general}, and let $c_1=c_1(k,t)$ be the constant of
Lemma~\ref{lem:semireg}. By Lemma~\ref{lem:semireg}, $\beta_{k,t}(G)\le c_1/v$ and
$\alpha_{k,t}(G)\le c_1/v+c_1\nu_t(G)/\binom{v}{t}$. Let
$v\ge\max\{k,c_1/\delta,2c_1/\varepsilon\}$. Then
$\beta_{k,t}(G)\le\delta$ and $c_1/v\le\varepsilon/2$, and
Theorem~\ref{thm:general} gives
\[
 P_G(v,k,t)\ \ge\ \Bigl(1-\varepsilon
 -c_1\,\nu_t(G)\Big/\binom{v}{t}\Bigr)\binom{v}{t}\Big/\binom{k}{t}.
\]
This is the assertion with $c=c_1$.
\end{proof}

\section{Applications}\label{sec:applications}

For OOCs, multidimensional OOCs and optical orthogonal signature pattern
codes, the maximum size of a code is expressed by $P_G(v,k,t)$, for a
semiregular group $G$ with at most three involutions. For quasi-cyclic
constant-weight codes, $P_G(v,k,t)$ is the maximum size of a code in which
every codeword has $|G|$ distinct images under $G$. For OOCs with at most
one pulse per wavelength, the maximum size is expressed by $P_G(\cF,t)$,
where $\cF$ is a set of transversals. For each
code family, Table~\ref{tab:summary} lists the group $G$, the set $V$, and
the parameters for which the Johnson bound is asymptotically attained, in
earlier work and in this paper. For semiregular groups with at most three
involutions, Theorem~\ref{thm:geh} gives the asymptotic value of
$P_G(v,k,t)$.

\begin{corollary}\label{cor:few}
Let $k>t\ge 2$ be integers and let $\varepsilon>0$. There exists an integer
$v_1=v_1(k,t,\varepsilon)$ such that
\[
 (1-\varepsilon)\binom{v}{t}\Big/\binom{k}{t}\ \le\ P_G(v,k,t)\ \le\
 \binom{v}{t}\Big/\binom{k}{t},
\]
for every semiregular permutation group $G$ with at most three involutions on
a set of $v\ge v_1$ points.
\end{corollary}

\begin{proof}
The upper bound is \eqref{eq:trivial}. By Lemma~\ref{lem:nu},
$\nu_2(G)=\iota(G)v/2\le 3v/2$, and $\nu_t(G)\le v^{t-1}$, for $t\ge 3$. Hence
$\nu_t(G)\le 2v^{t-1}$. Let $c$ and $v_0$ be the numbers given by
Theorem~\ref{thm:geh} for $\varepsilon/2$. For $v\ge v_0$, Theorem~\ref{thm:geh} gives
$P_G(v,k,t)\ge\bigl((1-\varepsilon/2)\binom{v}{t}-2cv^{t-1}\bigr)/\binom{k}{t}$.
Since $\binom{v}{t}\ge(v/t)^t$, we have
$2cv^{t-1}\le(\varepsilon/2)\binom{v}{t}$, for all
$v\ge 4ct^{t}/\varepsilon$.
\end{proof}

\begin{table}[ht]
\centering
\footnotesize
\renewcommand{\arraystretch}{1.25}
\begin{tabular}{@{}>{\raggedright\arraybackslash}p{4.2cm}>{\raggedright\arraybackslash}p{3.0cm}>{\raggedright\arraybackslash}p{3.6cm}>{\raggedright\arraybackslash}p{3.5cm}@{}}
\toprule
& & \multicolumn{2}{c@{}}{Johnson bound asymptotically attained for}\\
\cmidrule{3-4}
Code family & Group $G$ and set $V$ & Earlier work & This paper\\
\midrule
$(v,w,\lambda)$ OOCs & $\ZZ_v$ on $\ZZ_v$ &
lengths in particular congruence classes or of particular algebraic form \cite{AldersonMellinger,FujiHaraMiao,HuczynskaNg,Moreno,Omrani,Yin,ZhaoEtAl} &
all $v$\\\addlinespace[6pt]
$d$-dimensional $(\Lambda_1\times\dots\times\Lambda_{d-1}\times T,w,\lambda)$ OOCs, where $\Lambda=\Lambda_1\cdots\Lambda_{d-1}$ & $\ZZ_T$ on a set of size $\Lambda T$ &
particular $\Lambda$ and $T$ \cite{Alderson,Omrani2D} &
all $\Lambda$ and $T$\\\addlinespace[6pt]
Two-dimensional $(\Lambda\times T,w,\lambda)$ OOCs with at most one pulse per wavelength & $\ZZ_T$ on $\ZZ_\Lambda\times\ZZ_T$ &
$T$ prime, $T\ge w$ and $\Lambda\to\infty$ \cite{Omrani2D}\newline $w=3$, $\lambda=1$, all $T$ and $\Lambda\to\infty$ \cite{GallantJiangLing,Jiang,WangYin} &
all $\Lambda\ge w$ and all $T$\\\addlinespace[6pt]
$(m,n,w,\lambda)$ optical orthogonal signature pattern codes & $\ZZ_m\times\ZZ_n$ on itself &
particular $m$, $n$ and $w$ \cite{JiDingWangGe} &
all $m$ and $n$\\\addlinespace[6pt]
Binary quasi-cyclic codes of length $n$, index $\ell$, constant weight $w$ and minimum distance at least $2\delta$ & subgroup of order $n/\ell$ of $\ZZ_n$ on $\ZZ_n$ &
$\ell=n$ \cite{Rodl}\newline $\ell=1$ and the lengths $n$ in the first row &
all $n$ and all divisors $\ell$ of $n$\\
\bottomrule
\end{tabular}
\caption{Asymptotic attainment of the Johnson bound for constant weights and
constant $\lambda\ge 1$, with all weights at least $\lambda+2$. In the last
row, $\lambda=w-\delta$, and the codes are those in which every codeword has
$n/\ell$ distinct images under $G$.}
\label{tab:summary}
\end{table}

\subsection{Optical Orthogonal Codes}\label{sec:ooc}

Let $v>w>\lambda\ge 1$ be integers. As in Section~\ref{sec:introduction}, we identify a
codeword $x$ of an OOC with its support $C$. Let $C'$ be the support of a
codeword $y$. The sum in \eqref{eq:auto} equals $|C\cap(C-a)|$, and the sum
in \eqref{eq:cross} equals $|C\cap(C'-a)|$. Thus a
$(v,w,\lambda)$ OOC is a family $\cC$ of $w$-subsets of $\ZZ_v$ such
that
\[
 |C\cap(C+a)|\le\lambda,\quad\text{for all } C\in\cC \text{ and }
 a\in\ZZ_v\setminus\{0\},
\]
and
\[
 |C\cap(C'+a)|\le\lambda,\quad\text{for all distinct } C,C'\in\cC \text{ and }
 a\in\ZZ_v .
\]
These are the autocorrelation property and the cross-correlation property,
stated for supports.
Let $G=\ZZ_v$ act on $V=\ZZ_v$ by translation. This action is regular. The
following lemma is known \cite{FujiHaraMiao,Yin}.

\begin{lemma}\label{lem:ooc-packing}
Let $\cC$ be a family of $w$-subsets of $\ZZ_v$, and put
$\cB=\{C+a : C\in\cC,\ a\in\ZZ_v\}$. Then $\cC$ is a $(v,w,\lambda)$ OOC if
and only if $\cB$ is a free $G$-invariant $(\lambda+1)$-$(v,w,1)$ packing
with $v|\cC|$ blocks. Moreover, every free $G$-invariant
$(\lambda+1)$-$(v,w,1)$ packing on $\ZZ_v$ has this form, where $\cC$ is an
arbitrary set of representatives of the orbits of $G$ on its blocks.
\end{lemma}

\begin{proof}
The family $\cB$ is $G$-invariant. Suppose that $\cC$ is a $(v,w,\lambda)$
OOC. Let $C,C'\in\cC$ and $a,b\in\ZZ_v$ with $(C,a)\ne(C',b)$. Then
$|(C+a)\cap(C'+b)|=|C\cap(C'+b-a)|\le\lambda$. This follows from the autocorrelation property if $C=C'$, because then
$b-a\ne 0$, and from the cross-correlation property if $C\ne C'$. Since $\lambda<w$, we have $C+a\ne C'+b$. Hence $\cB$ has
$v|\cC|$ blocks, every block is free, and every two distinct blocks have at
most $\lambda$ common points. Thus $\cB$ is a free $G$-invariant
$(\lambda+1)$-$(v,w,1)$ packing.

Suppose now that $\cB$ is a free $G$-invariant $(\lambda+1)$-$(v,w,1)$
packing with $v|\cC|$ blocks. Then the sets $C+a$ with $C\in\cC$ and
$a\in\ZZ_v$ are pairwise distinct. Two distinct blocks of a
$(\lambda+1)$-$(v,w,1)$ packing have at most $\lambda$ common points. Hence
$|C\cap(C+a)|\le\lambda$, for all $C\in\cC$ and $a\ne 0$, and
$|C\cap(C'+a)|\le\lambda$, for all distinct $C,C'\in\cC$ and all $a$. Thus
$\cC$ is a $(v,w,\lambda)$ OOC.

Finally, let $\cB'$ be a free $G$-invariant $(\lambda+1)$-$(v,w,1)$ packing
on $\ZZ_v$, and let $\cC$ be a set of representatives of the orbits of $G$ on
$\cB'$. Every orbit has $v$ blocks. Hence
$\cB'=\{C+a : C\in\cC,\ a\in\ZZ_v\}$, and $\cB'$ has $v|\cC|$ blocks.
\end{proof}

\begin{corollary}\label{cor:ooc}
Let $G=\ZZ_v$ act on $\ZZ_v$ by translation. Then
\[
 \Phi(v,w,\lambda)=\frac{1}{v}\,P_G(v,w,\lambda+1).
\]
Consequently, for all integers $\lambda\ge 1$ and $w\ge\lambda+2$,
\[
 \Phi(v,w,\lambda)=(1-o(1))\,
 \frac{(v-1)(v-2)\cdots(v-\lambda)}{w(w-1)\cdots(w-\lambda)}
 \qquad (v\to\infty).
\]
\end{corollary}

\begin{proof}
The identity $\Phi(v,w,\lambda)=P_G(v,w,\lambda+1)/v$ is
Lemma~\ref{lem:ooc-packing}. The group $G$ is cyclic, so it has at most one
involution. The asymptotic formula follows from
Corollary~\ref{cor:few} with $t=\lambda+1$ and $k=w$, because
$\binom{v}{\lambda+1}/(v\binom{w}{\lambda+1})$ is the right-hand side of
\eqref{eq:ooc-upper}.
\end{proof}

Let $d\ge 2$ and let
$\Lambda_1,\dots,\Lambda_{d-1},T$ be positive integers.
Put $\Lambda=\Lambda_1\cdots\Lambda_{d-1}$ and
\[
 V=\ZZ_{\Lambda_1}\times\dots\times\ZZ_{\Lambda_{d-1}}\times\ZZ_T .
\]
The group $\ZZ_T$ acts on $V$ by translation of the last coordinate. For
$C\subseteq V$ and $a\in\ZZ_T$, write $C+a$ for the image of $C$. A
\emph{$(\Lambda_1\times\dots\times\Lambda_{d-1}\times T,w,\lambda)$ OOC} is a family $\cC$ of $w$-subsets of $V$ that has the autocorrelation
property and the cross-correlation property, stated for supports, with $a$
ranging over $\ZZ_T$.
In the applications, the coordinates for $d=2$ represent wavelength and time,
and for $d=3$ there is an additional coordinate for space or polarisation. For $d=2$, Omrani, Garg, Kumar, Elia and
Bhambhani \cite{Omrani2D} proved the Johnson bound
\begin{equation}\label{eq:johnson-nd}
 |\cC|\le\left\lfloor\frac{\Lambda}{w}\left\lfloor\frac{\Lambda T-1}{w-1}
 \left\lfloor\frac{\Lambda T-2}{w-2}\cdots\left\lfloor
 \frac{\Lambda T-\lambda}{w-\lambda}
 \right\rfloor\cdots\right\rfloor\right\rfloor\right\rfloor
\end{equation}
and constructed optimal and asymptotically optimal families for particular
$\Lambda$ and $T$ from polynomials and rational functions over finite fields.
Alderson \cite{Alderson} extended \eqref{eq:johnson-nd} to all $d$ and gave
further families from finite projective spaces.

The action of $\ZZ_T$ on $V$ is semiregular with $m=T$ and $v=\Lambda T$. This action is regular only when $\Lambda=1$.

\begin{corollary}\label{cor:nd}
Let $G=\ZZ_T$ act on $V$ by translation of the last coordinate. The largest
size of a $(\Lambda_1\times\dots\times\Lambda_{d-1}\times T,w,\lambda)$ OOC
is
\[
 \frac{1}{T}\,P_G(\Lambda T,w,\lambda+1).
\]
Consequently, for all integers $\lambda\ge 1$ and $w\ge\lambda+2$, this
largest size is
\[
 (1-o(1))\,\frac{\Lambda\,(\Lambda T-1)(\Lambda T-2)\cdots(\Lambda T-\lambda)}
 {w(w-1)\cdots(w-\lambda)}
 \qquad(\Lambda T\to\infty).
\]
The estimate is uniform in $d$ and in $\Lambda_1,\dots,\Lambda_{d-1},T$.
\end{corollary}

\begin{proof}
The proof of Lemma~\ref{lem:ooc-packing} applies with $\ZZ_T$ in place of
$\ZZ_v$ and $V$ in place of $\ZZ_v$. This proof shows that a family $\cC$ of $w$-subsets of $V$ is a $(\Lambda_1\times\dots\times\Lambda_{d-1}\times T,w,\lambda)$ OOC if and only
if the sets $C+a$, with $C\in\cC$ and $a\in\ZZ_T$, form a free $G$-invariant
$(\lambda+1)$-$(\Lambda T,w,1)$ packing with $T|\cC|$ blocks. It also shows
that every free $G$-invariant $(\lambda+1)$-$(\Lambda T,w,1)$ packing
consists of the sets $C+a$, with $C\in\cC$ and $a\in\ZZ_T$, for a family
$\cC$ of $w$-subsets of $V$. Hence the largest size of a code is
$P_G(\Lambda T,w,\lambda+1)/T$. The group $G$ is cyclic, and the asymptotic
formula follows from Corollary~\ref{cor:few} with $v=\Lambda T$, $t=\lambda+1$ and
$k=w$.
\end{proof}

Thus \eqref{eq:johnson-nd} is asymptotically tight for every $d$.
Let
$m$ and $n$ be positive integers. An \emph{$(m,n,w,\lambda)$ optical
orthogonal signature pattern code} is a family $\cC$ of $w$-subsets of
$\ZZ_m\times\ZZ_n$ that has the autocorrelation property and the
cross-correlation property, stated for supports, with $a$ ranging over
$\ZZ_m\times\ZZ_n$. Optical orthogonal signature pattern codes were
introduced by Kitayama \cite{Kitayama}
for parallel transmission of images. Let $\Psi(m,n,w,\lambda)$ be the
largest size of an $(m,n,w,\lambda)$ optical orthogonal signature pattern
code. Ji, Ding, Wang and Ge
\cite{JiDingWangGe} constructed asymptotically optimal families for
particular values of $m$, $n$ and $w$.

\begin{corollary}\label{cor:oospc}
Let $G=\ZZ_m\times\ZZ_n$ act on itself by translation. Then
\[
 \Psi(m,n,w,\lambda)=\frac{1}{mn}\,P_G(mn,w,\lambda+1).
\]
Consequently, for all integers $\lambda\ge 1$ and $w\ge\lambda+2$,
\[
 \Psi(m,n,w,\lambda)=(1-o(1))\,
 \frac{(mn-1)(mn-2)\cdots(mn-\lambda)}{w(w-1)\cdots(w-\lambda)}
 \qquad(mn\to\infty).
\]
\end{corollary}

\begin{proof}
The proof of Lemma~\ref{lem:ooc-packing} applies with $\ZZ_m\times\ZZ_n$ in
place of $\ZZ_v$, and gives the identity
$\Psi(m,n,w,\lambda)=P_G(mn,w,\lambda+1)/(mn)$. The group $G$ acts regularly
and has at most three involutions. The asymptotic formula follows from
Corollary~\ref{cor:few} with $v=mn$, $t=\lambda+1$ and $k=w$.
\end{proof}

More generally, let $A$ be a group of order $v$, and consider families of
$w$-subsets of $A$ that have the autocorrelation property and the
cross-correlation property with respect to left translation by $A$. For
$\lambda\ge 2$ and $w\ge\lambda+2$, the largest size of a family with these
two properties is
$(1-o(1))\binom{v}{\lambda+1}/(v\binom{w}{\lambda+1})$, for every group $A$.
For $\lambda=1$ and $w\ge 3$, the largest size is
$(1-o(1))\binom{v}{2}/(v\binom{w}{2})$ if and only if $A$ has $o(v)$
involutions, by Theorem~\ref{thm:geh} and Proposition~\ref{prop:pairs}.
Example~\ref{ex:elementary} shows that, for $\lambda=1$ and $w\ge 2$, no
nonempty family of $w$-subsets of an elementary abelian $2$-group has the two
properties. Huczynska and Ng \cite{HuczynskaNg} ask for constructions of
OOCs over arbitrary groups, in particular over nonabelian
groups.

\subsection{Cyclic and Quasi-Cyclic Constant-Weight Codes}\label{sec:cw}

A binary word of length $n$ and weight $w$ is identified with its support, a
$w$-subset of $\ZZ_n$. Two words of weight $w$ are at Hamming distance at
least $2\delta$ if and only if the intersection of their supports has at most
$w-\delta$ points. Hence a binary code of length $n$, constant weight $w$ and minimum
distance at least $2\delta$ is a $(w-\delta+1)$-$(n,w,1)$ packing. Let
$A(n,2\delta,w)$ be the largest size of a binary code of length $n$,
constant weight $w$ and minimum distance at least $2\delta$. The bound
\eqref{eq:trivial} is the first-order form of the Johnson bound
\cite{Johnson}, and R\"odl's theorem shows that it is asymptotically attained.
Liu and Shangguan \cite{LiuShangguan} used matchings in hypergraphs to
determine the asymptotic maximum size of $q$-ary constant-weight codes of odd
minimum distance.

Let $\ell$ be a divisor of $n$. A code is \emph{quasi-cyclic of index $\ell$}
if it is invariant under the cyclic shift by $\ell$ positions. The case
$\ell=1$ is that of cyclic codes. The shift by $\ell$ positions generates a
semiregular group $G$ of order $n/\ell$ on $\ZZ_n$.

\begin{corollary}\label{cor:cyclic}
Let $w$ and $\delta$ be integers with $2\le\delta\le w-1$, let $\ell$ be a
divisor of $n$, and let $G$ be the group generated by the cyclic shift of
$\ZZ_n$ by $\ell$ positions. The maximum number of codewords in a binary
quasi-cyclic code of index $\ell$, length $n$, constant weight $w$ and
minimum distance at least $2\delta$, in which every codeword has $n/\ell$
distinct images under $G$, is
\[
 P_G(n,w,w-\delta+1).
\]
Consequently, this maximum is
\[
 (1-o(1))\binom{n}{w-\delta+1}\Big/\binom{w}{w-\delta+1}
 =(1-o(1))\,A(n,2\delta,w)
 \qquad(n\to\infty),
\]
uniformly over all divisors $\ell$ of $n$.
\end{corollary}

\begin{proof}
Put $t=w-\delta+1$. A binary quasi-cyclic code of index $\ell$, length $n$,
constant weight $w$ and minimum distance at least $2\delta$, in which every
codeword has $n/\ell$ distinct images under $G$, is a $G$-invariant
$t$-$(n,w,1)$ packing in which every block has trivial stabiliser. Hence it
is a free $G$-invariant $t$-$(n,w,1)$ packing, and every free $G$-invariant
$t$-$(n,w,1)$ packing is a code with these properties. Therefore the maximum
number of codewords is $P_G(n,w,t)$. The group $G$ is cyclic, and
$w>t\ge 2$. The asymptotic formula follows from Corollary~\ref{cor:few} with $v=n$ and $k=w$, and from
R\"odl's theorem for $A(n,2\delta,w)$.
\end{proof}

For $\ell=1$, a set of orbit representatives of the code in
Corollary~\ref{cor:cyclic} is a cyclically permutable constant-weight code in
the sense of Nguyen, Gy\"orfi and Massey \cite{NguyenGyorfiMassey}.
By Lemma~\ref{lem:ooc-packing}, it is also an $(n,w,w-\delta)$ OOC, and
$P_G(n,w,w-\delta+1)=n\,\Phi(n,w,w-\delta)$. Hence, for $\ell=1$, the
asymptotic formula in Corollary~\ref{cor:cyclic} was known for every sequence
of lengths $n$ for which asymptotically optimal $(n,w,w-\delta)$ OOCs were
known. For example, the cyclic difference families of Zhao et al.\
\cite[Theorem~1.2]{ZhaoEtAl} give an $(n,4,1)$ OOC with
$\lfloor(n-1)/12\rfloor$ codewords for all $n\ge 14$ with
$n\equiv 2,3\pmod{12}$ and $n\ne 27$. For $\ell=n$, the group
$G$ is trivial, and the asymptotic formula is R\"odl's theorem.
Corollary~\ref{cor:cyclic} gives the asymptotic formula for all $w$, all
$\delta$ with $2\le\delta\le w-1$, all $n$ and all divisors $\ell$ of $n$.

A quasi-cyclic code of index $\ell$ may contain codewords with fewer than
$n/\ell$ distinct images under $G$. Let $Q$ be the maximum number of
codewords in a binary quasi-cyclic code of index $\ell$, length $n$,
constant weight $w$ and minimum distance at least $2\delta$. Then
$P_G(n,w,w-\delta+1)\le Q\le A(n,2\delta,w)$. By
Corollary~\ref{cor:cyclic}, we have $Q=(1-o(1))\,A(n,2\delta,w)$ as
$n\to\infty$, uniformly over all divisors $\ell$ of $n$.

\subsection{Group Divisible Packings and OOCs with At Most One Pulse per Wavelength}\label{sec:gd}

A $t$-$(v,k,1)$ packing is a packing of the $t$-subsets of $V$ by
$k$-subsets. OOCs with at most one pulse per wavelength correspond to packings
of a different kind. Let $\Lambda$ and $T$ be positive integers, and put
$V=\ZZ_\Lambda\times\ZZ_T$ and $v=\Lambda T$. The sets $\{i\}\times\ZZ_T$,
$i\in\ZZ_\Lambda$, are the \emph{parts} of $V$. A subset of $V$ is a
\emph{transversal} if it has at most one point in each part. Let
$k\ge t\ge 1$ be integers. A \emph{group divisible $t$-packing} on $V$ with
block size $k$ is a set $\cB$ of transversal $k$-subsets of $V$, called
blocks, such that every transversal $t$-subset of $V$ lies in at most one
block. The term comes from design theory, where the parts are called groups.
Every $t$-subset of a block is a transversal, and there are
$\binom{\Lambda}{t}T^{t}$ transversal $t$-subsets. Hence
\begin{equation}\label{eq:gd-bound}
 |\cB|\le\binom{\Lambda}{t}T^{t}\Big/\binom{k}{t}.
\end{equation}
A group divisible $t$-packing with block size $k$ is a $t$-$(v,k,1)$ packing,
all of whose blocks are transversals, and conversely.

Let $G=\ZZ_T$ act on $V$ by translation of the second coordinate. For
$A\subseteq V$ and $a\in\ZZ_T$, write $A+a$ for the image of $A$. This action
is semiregular, and it maps transversals to transversals. A transversal $A$
with at least one point satisfies $A\cap(A+a)=\emptyset$, for all
$a\in\ZZ_T\setminus\{0\}$. Hence every nonempty transversal has trivial stabiliser,
and every $G$-invariant group divisible packing is free. Let $\cF_k$ be the
set of all transversal $k$-subsets of $V$. Then $P_G(\cF_k,t)$ is the maximum
number of blocks in a $G$-invariant group divisible $t$-packing on $V$ with
block size $k$.

\begin{theorem}[$\ZZ_T$-invariant group divisible packings]\label{thm:gd}
Let $k>t\ge 2$ be integers. Then
\[
 P_G(\cF_k,t)=(1-o(1))\binom{\Lambda}{t}T^{t}\Big/\binom{k}{t}
 \qquad(\Lambda T\to\infty),
\]
uniformly over all $\Lambda\ge k$ and all $T\ge 1$.
\end{theorem}

\begin{proof}
The upper bound is \eqref{eq:gd-bound}. For the lower bound, we apply Proposition~\ref{prop:restricted} to $\cF=\cF_k$. We
have $|\cF|=\binom{\Lambda}{k}T^{k}$, and $B+a\in\cF$, for all
$B\in\cF$ and all $a\in\ZZ_T$. Every $B\in\cF$ satisfies
$B\cap(B+a)=\emptyset$, for all $a\in\ZZ_T\setminus\{0\}$, so it is
$t$-separated. A transversal of size $j\le k$ is included in exactly
$\binom{\Lambda-j}{k-j}T^{k-j}$ members of $\cF$, and a subset of $V$ that is
not a transversal is included in none. Put
\[
 D=\binom{\Lambda-t}{k-t}T^{k-t}.
\]
Then every $t$-subset of $V$ is included in at most $D$ members of $\cF$. For
$t\le j<k$, the ratio of $\binom{\Lambda-j-1}{k-j-1}T^{k-j-1}$ to
$\binom{\Lambda-j}{k-j}T^{k-j}$ is $(k-j)/((\Lambda-j)T)$. Since
$\Lambda\ge k$, this ratio is at most $\rho=(k-t)/((\Lambda-t)T)$, and
$\rho\le k/v$. Hence a subset of $V$ of size $j\ge t$ is included in at most
$\rho^{j-t}D$ members of $\cF$.

Let $T_0$ and $S$ be $t$-subsets of $V$ in distinct orbits of $\ZZ_T$. At most
$t^2$ elements $a\in\ZZ_T$ satisfy $(S+a)\cap T_0\ne\emptyset$. For each of
them $|T_0\cup(S+a)|\ge t+1$, and at most $\rho D$ members of $\cF$ include
$T_0\cup(S+a)$. At most $T$ elements $a$ satisfy
$(S+a)\cap T_0=\emptyset$. For each of them $|T_0\cup(S+a)|=2t$, and at most
$\rho^{t}D$ members of $\cF$ include $T_0\cup(S+a)$. Now
$T\rho^{t}\le k^{t}/(\Lambda^{t}T^{t-1})\le k^{t}/v$, because $t\ge 2$.
Hence at most $(t^{2}k+k^{t})D/v$ members of $\cF$ include $T_0$ and an image
of $S$.

Let $\varepsilon>0$, and let $\delta$ be the number given by
Theorem~\ref{thm:kahn} for $s=\binom{k}{t}$ and $\varepsilon$. Let
$v\ge(t^{2}k+k^{t})/\delta$. Proposition~\ref{prop:restricted} gives
\[
 P_G(\cF_k,t)\ \ge\ \frac{\binom{\Lambda}{k}T^{k}}{(1+\varepsilon)D}
 =\frac{1}{1+\varepsilon}\binom{\Lambda}{t}T^{t}\Big/\binom{k}{t}.
\]
\end{proof}

Theorem~\ref{thm:gd} has no condition on the stabilisers of $t$-subsets,
because every nonempty transversal has trivial stabiliser. The proof of
Theorem~\ref{thm:gd} uses only the action of $\ZZ_T$ on the second coordinate.
Hence Theorem~\ref{thm:gd} holds with
$\ZZ_\Lambda$ replaced by an arbitrary set of $\Lambda$ elements.

A $(\Lambda\times T,w,\lambda)$ OOC, as defined in
Section~\ref{sec:ooc} with $d=2$, has \emph{at most one pulse per wavelength}
if every codeword is a transversal. The coordinate in $\ZZ_\Lambda$ represents
the wavelength. A $w$-subset $C$ of $V$ is a transversal if and only if
$C\cap(C+a)=\emptyset$, for all $a\in\ZZ_T\setminus\{0\}$. Hence the OOCs
with at most one pulse per wavelength are the $(\Lambda\times T,w,\lambda)$
OOCs with autocorrelation zero.
Let $\Phi^{*}(\Lambda\times T,w,\lambda)$ be the maximum size of a
$(\Lambda\times T,w,\lambda)$ OOC with at most one pulse per wavelength.

\begin{corollary}\label{cor:amopp}
We have
\[
 \Phi^{*}(\Lambda\times T,w,\lambda)=\frac{1}{T}\,P_G(\cF_w,\lambda+1).
\]
Consequently, for all integers $\lambda\ge 1$ and $w\ge\lambda+2$,
\[
 \Phi^{*}(\Lambda\times T,w,\lambda)=(1-o(1))\,
 \binom{\Lambda}{\lambda+1}T^{\lambda}\Big/\binom{w}{\lambda+1}
 \qquad(\Lambda T\to\infty),
\]
uniformly over all $\Lambda\ge w$ and all $T\ge 1$.
\end{corollary}

\begin{proof}
The proof of Lemma~\ref{lem:ooc-packing} applies with $\ZZ_T$ in place of
$\ZZ_v$ and $V$ in place of $\ZZ_v$. This proof shows that a family $\cC$ of transversal $w$-subsets of $V$ is a code with at most one pulse per
wavelength if and only if the sets $C+a$, with $C\in\cC$ and $a\in\ZZ_T$,
form a $G$-invariant group divisible $(\lambda+1)$-packing with block size
$w$ and with $T|\cC|$ blocks. It also shows that every $G$-invariant group
divisible $(\lambda+1)$-packing with block size $w$ consists of the sets
$C+a$, with $C\in\cC$ and $a\in\ZZ_T$, for a family $\cC$ of transversal
$w$-subsets of $V$. This proves the identity
$\Phi^{*}(\Lambda\times T,w,\lambda)=P_G(\cF_w,\lambda+1)/T$. The asymptotic
formula follows from Theorem~\ref{thm:gd} with
$t=\lambda+1$ and $k=w$.
\end{proof}

The upper bound in Corollary~\ref{cor:amopp} is the Johnson bound of
\cite{Omrani2D} for codes with at most one pulse per wavelength, without the
floor functions. The lower bound was known for the following parameters.
Let $T$ be a prime
with $T\ge w$. From a binary code of length $\Lambda$, constant weight $w$ and
minimum distance $2(w-\lambda)$ with $M$ codewords, the concatenation
construction of \cite{Omrani2D} gives a code with at most one pulse per
wavelength and $T^{\lambda}M$ codewords. By R\"odl's theorem, this proves the
lower bound in Corollary~\ref{cor:amopp} for primes $T\ge w$ and
$\Lambda\to\infty$.

Let $w=3$ and $\lambda=1$. A \emph{semicyclic group
divisible design} with block size three and with $\Lambda'$ parts of size $T$
is a $G$-invariant group divisible $2$-packing on
$\ZZ_{\Lambda'}\times\ZZ_T$ with block size three in which every transversal
$2$-subset lies in a block. Such a design has $\binom{\Lambda'}{2}T^{2}/3$ blocks. Wang and Yin \cite[Theorem~6]{WangYin} proved that it is equivalent to a
$(\Lambda'\times T,3,1)$ OOC with at most one pulse per wavelength and
$\binom{\Lambda'}{2}T/3$ codewords. Gallant, Jiang
and Ling \cite[Theorem~2]{GallantJiangLing} determined the parameters
$\Lambda'$ and $T$ for which such a design exists, and the case $T\equiv 0\pmod 6$ is
due to Jiang \cite{Jiang}. By their theorem, such a design exists for all $T\ge 1$ and
all $\Lambda'\ge 13$ with $\Lambda'\equiv 1\pmod{12}$. Let $\Lambda\ge 13$, and
let $\Lambda'$ be the largest integer with $\Lambda'\le\Lambda$ and
$\Lambda'\equiv 1\pmod{12}$. Then $\Lambda-\Lambda'\le 11$, and hence
\[
 \Phi^{*}(\Lambda\times T,3,1)\ \ge\ \binom{\Lambda'}{2}\frac{T}{3}
 =\bigl(1-O(1/\Lambda)\bigr)\binom{\Lambda}{2}\frac{T}{3},
\]
uniformly over all $T\ge 1$. This proves the lower bound in
Corollary~\ref{cor:amopp} for $w=3$, $\lambda=1$ and $\Lambda\to\infty$.

Corollary~\ref{cor:amopp} holds for all $\lambda\ge 1$ and all
$w\ge\lambda+2$, it has no condition on $T$, and it includes the case in
which $\Lambda$ is bounded and $T\to\infty$. By the
remark after Theorem~\ref{thm:gd}, the corollary applies to codes of every
dimension with autocorrelation zero. Alderson \cite{AldersonIdeal,Alderson}
proved the corresponding Johnson bound for codes of every dimension with
autocorrelation zero, with $\Lambda$ equal
to the product of the sizes of the coordinates other than time.

\section{Concluding Remarks}

The existence of packings and designs with a prescribed group of
automorphisms is a classical problem of design theory.
Theorem~\ref{thm:geh} gives asymptotic equality in \eqref{eq:trivial} for
free $G$-invariant packings, for every semiregular group $G$ with
$\nu_t(G)=o(v^t)$. The existence of Steiner systems with a prescribed
semiregular group of automorphisms is open in general.

Theorem~\ref{thm:general} applies to groups that are not semiregular.
An example is the dihedral group $G$ of order $2v$ acting on $\ZZ_v$. For all $x,y\in\ZZ_v$, exactly two elements of $G$ map $x$ to $y$. The
argument for $\beta_{k,t}(G)$ in the proof of Lemma~\ref{lem:semireg} uses
only a bound on the number of elements that map $x$ to $y$, and it gives
$\beta_{k,t}(G)\le(2t^2k+2k^{t})/v$. Let $t\ge 3$. The rotations form a
regular cyclic group, and their contribution to \eqref{eq:alpha-sum} is
$O(1/v)$ by Lemmas~\ref{lem:nu} and~\ref{lem:semireg}. A reflection $g$ has at most two points $x$ with $gx=x$, and all its other cycles have length two. Let $T$
be a $t$-subset of $\ZZ_v$. Let $a$ be the number of points $x\in T$ with
$gx=x$, let $b$ be the number of cycles of $g$ of length two that are
included in $T$, and let $c$ be the number of the remaining points of $T$.
Then $a+2b+c=t$ and $|T\cup gT|=a+2b+2c$. For given $a$, $b$ and $c$, there
are $O(v^{b+c})$ sets $T$, and each contributes at most $(k/v)^{a+2b+2c}$ to
\eqref{eq:alpha-sum} by \eqref{eq:prob}. Summation over the $v$ reflections
gives $O(v^{1-a-b-c})$. Since $t\ge 3$, we have $a+b+c\ge 2$. Hence $\alpha_{k,t}(G)=O(1/v)$, for $t\ge 3$. By
Theorem~\ref{thm:general}, we obtain
$P_G(v,k,t)=(1-o(1))\binom{v}{t}/\binom{k}{t}$ as $v\to\infty$, for all
constant $k>t\ge 3$. In contrast, for $t=2$, every
$2$-subset $\{x,y\}$ of $\ZZ_v$ is mapped to itself by the reflection
$z\mapsto x+y-z$. Hence $\nu_2(G)=\binom{v}{2}$, and $P_G(v,k,2)=0$ by
Proposition~\ref{prop:pairs}.

Whether the hypothesis on $\beta_{k,t}(G)$ in
Theorem~\ref{thm:general} is necessary is an open problem. Equivalently, the problem is whether $\alpha_{k,t}(G)\to 0$ alone implies $P_G(v,k,t)=(1-o(1))\binom{v}{t}/\binom{k}{t}$.

\section*{Acknowledgement}

Research supported by SUTD Grant SKI 2021\_07\_04.

\end{document}